\documentclass{amsart}
\usepackage{amsmath}
\usepackage{amsfonts}
\usepackage{amsthm, upref}
\usepackage{graphicx}
\usepackage{enumerate}
\usepackage[usenames, dvipsnames]{color}

\input xy
\xyoption{all}

\usepackage{ifthen}
\usepackage{tikz}
\usepackage{appendix}
\usepackage{verbatim}
\usetikzlibrary{decorations.pathmorphing}
\usetikzlibrary{calc}
\usepackage{hyperref}

\renewcommand{\comment}[1]{}
\newcommand{\eq}{\begin{equation}}
\newcommand{\en}{\end{equation}}
\newcommand{\rr}{\mathbb{R}}
\newcommand{\rbar}{\overline{\rr}}
\newcommand{\NN}{\mathbb{N}}

\newcommand{\norm}[1]{\left\lVert #1 \right\rVert}
\newcommand{\abs}[1]{\left\lvert #1 \right\rvert}

\newcommand{\iprod}[1]{\left\langle #1 \right\rangle }

\newcommand{\VS}{\mathcal{V}}
\newcommand{\topo}{\mathcal{E}}
\newcommand{\omap}{\mathbf{s}}
\newcommand{\dom}{\mathrm{dom}}

\begin{document}

\theoremstyle{plain}
\newtheorem{thm}{Theorem}
\newtheorem{lemma}[thm]{Lemma}
\newtheorem{prop}[thm]{Proposition}
\newtheorem{cor}[thm]{Corollary}

\theoremstyle{definition}
\newtheorem{defn}{Definition}
\newtheorem{asmp}{Assumption}
\newtheorem{notn}{Notation}
\newtheorem{prb}{Problem}

\theoremstyle{remark}
\newtheorem{rmk}{Remark}
\newtheorem{exm}{Example}
\newtheorem{clm}{Claim}

\title[Embedding optimal transports]{Embedding optimal transports in statistical manifolds}

\author{Soumik Pal}
\address{Department of Mathematics\\ University of Washington\\ Seattle, WA 98195}
\email{soumikpal@gmail.com}



\thanks{This research is partially supported by NSF grant DMS-1308340 and DMS-1612483}

\date{\today}

\begin{abstract} We consider Monge-Kantorovich optimal transport problems on $\rr^d$, $d\ge 1$, with a convex cost function given by the cumulant generating function of a probability measure. Examples include the Wasserstein-$2$ transport whose cost function is the square of the Euclidean distance and corresponds to the cumulant generating function of the multivariate standard normal distribution. The optimal coupling is usually described via an extended notion of convex/concave functions and their gradient maps. These extended notions are nonintuitive and do not satisfy useful inequalities such as Jensen's inequality. Under mild regularity conditions, we show that all such extended gradient maps can be recovered as the usual supergradients of a nonnegative concave function on the space of probability distributions. This embedding provides a universal geometry for all such optimal transports and an unexpected connection with information geometry of exponential families of distributions.  
\end{abstract}

\maketitle

\section{Introduction} Fix $d\in \NN$, where $\NN$ is the set of natural numbers. Fix a Borel probability measure $\mu_0$ on $\rr^d$. This will be called the base measure. For any $\mu_0$ integrable function $f$, denote the $\mu_0$ expectation of $f$ by $\mu_0\left( f(x) \right)$ or $\mu_0(f)$. Let $\Lambda_0$ denote the cumulant generating function of $\mu_0$. That is, for any $\theta\in \rr^d$, we have 
\[
\Lambda_0(\theta):= \log \mu_0\left(  e^{\iprod{\theta, x}} \right). 
\]
We will throughout assume that its domain is the entire space, i.e., $\dom\left( \Lambda_0 \right)=\rr^d$. 

It is well-known that $\Lambda_0$ is a convex function, which is strict whenever the support of $\mu_0$ is not a singleton. Moreover, $\Lambda_0(0)=0$. For $\theta, \psi\in \rr^d$, we define a \textit{cost function} on $\rr^d \times \rr^d$ given by
\[
c(\theta, \psi):=\Lambda_0(\theta-\psi)=\log \mu_0\left(  e^{\iprod{\theta-\psi, x}} \right). 
\]

Suppose $P$ and $Q$ are two probability measures on $\rr^d$. Let $\Pi(P,Q)$ be the set of couplings of $(P,Q)$, i.e., the set of joint probability distributions on $\rr^d \times \rr^d$ whose first marginal is $P$ and the second is $Q$. Consider the Monge-Kantorovich optimal transport problem of transporting $P$ to $Q$ with cost $c$. That is, we find the minimizer in the optimization problem
\eq\label{eq:OTstatement}
\min_{R \in \Pi(P,Q)} R\left[  \Lambda_0(\theta-\psi)   \right], \quad (\theta, \psi) \sim R.
\en 
The optimal coupling $R$, if exists, is said to solve the Monge problem if $\psi$ is a deterministic function of $\theta$, $R$ almost surely. Let us consider two known examples.

\begin{exm}\label{exmp:gaussian}
Let $\mu_0\sim N(0,I)$ be the standard Gaussian distribution on $\rr^d$. In this case, $\Lambda_0(\theta-\psi)=\norm{\theta-\psi}^2/2$ and the optimal coupling between $P$ and $Q$ is the well-known Wasserstein-$2$ or $W_2$ coupling. 
\end{exm}

\begin{exm}\label{exmp:simplex}
Let $e_i, i\in [d]$, be the standard basis in $\rr^d$. Additionally, let $e_0$ denote the zero vector in $\rr^d$. Let $\mu_0$ denote the probability measure that puts mass $1/(d+1)$ at each $e_i$. That is, if $\delta$ refers to the unit Dirac delta mass, then 
\[
\mu_0= \frac{1}{d+1} \sum_{i=0}^d \delta_{e_i}. 
\]
Thus $\Lambda_0(\theta)=\log\left[ 1 + \sum_{i=1}^d e^{\theta_i}   \right] - \log(d+1)$. The corresponding optimal transport problem has appeared recently in \cite{PW14, P16, PW16} in connection with portfolio theory and information geometry. 
\end{exm}

The modern theory of optimal transport is a vast area with important applications in analysis, geometry, and probability. We refer the reader to the book \cite{AG13} for an introduction.

We now describe the solution to the optimization problem \eqref{eq:OTstatement} as done in \cite{GM96}. The optimal coupling can be described by an extension of the usual notion of concavity that we describe below. See, for example, \cite[Definition 1.8 and eqn. 1.3, Chapter 1]{AG13}.

\begin{defn}\label{defn:c-concave} A function $\psi:\rr^d\rightarrow \rr\cup \{-\infty\}$ is said to be $\Lambda_0$ concave if 
\[
\psi(x)=\inf_{y \in \rr^d}\left[ \Lambda_0(x-y) - \rho(y)  \right]
\]
for some $\rho:\rr^d \rightarrow \rr \cup \{-\infty\}$. The $\Lambda_0$ superdifferential of a function $\psi :\rr^d \mapsto \rr \cup \{-\infty\}$ at a point $\theta\in \rr^d$ is given by the set of points $(\theta, y)\in \rr^d \times \rr^d$ such that 
\[
\psi(v) \le \psi(\theta) + \Lambda_0(v-y) - \Lambda_0(\theta - y), \quad \text{for all $v \in \rr^d$}. 
\]
The set of $\Lambda_0$ superdifferential pairs $(\theta, y)$ of $\psi$ will be denoted by $\partial^{\Lambda_0}\psi$. In particular, $y \in \partial^{\Lambda_0}\psi(\theta)$ will mean $(\theta, y)\in \partial^{\Lambda_0}\psi$. 

Define the $\Lambda_0$-transform (or, conjugate) $\psi^{0}:\rr^d \rightarrow \rr \cup \{-\infty\}$ by 
\eq\label{eq:lambdaconj}
\psi^0(y)=\inf_{x \in \rr^d}\left[ \Lambda_0(x-y) - \psi(x)  \right].
\en
Then, $\psi^0$ is \textit{dual} $\Lambda_0$ concave, and $\psi(\theta)+\psi^0(y)=\Lambda_0(\theta - y)$ if and only if $(\theta,y)\in \partial^{\Lambda_0}\psi$. See \cite[Section 3.1]{PW16} for more details.
\end{defn}

The following theorem is taken from \cite[Theorem 1.13, Chapter 1]{AG13} where it is referred to as the \textit{Fundamental theorem of optimal transport}.  

\begin{thm} For a joint distribution $R\in \Pi(P,Q)$ to be optimal for the minimization problem \eqref{eq:OTstatement} it is sufficient that there exists a $\Lambda_0$ concave function $\psi$ such that $\max(\psi, 0)$ is $P$ integrable and the support of $R$ is a subset of $\partial^{\Lambda_0}\psi$. The above is also necessary if $\Lambda_0$ is bounded below. 
\end{thm}

In our context, $\Lambda_0$ is bounded below by zero if the mean of $\mu_0$ is zero. 
\smallskip

Although these extended notions of concavity and superdifferentiability gives us a theoretical picture of the optimal coupling, they can be nonintuitive and do not satisfy useful inequalities satisfied by the superdifferentials of ordinary concave functions. However, in this article we show that they can all be embedded as superdifferentials of an actual concave function on the space of probability distributions that are absolutely continuous with respect to $\mu_0$.  

To understand the idea, one first needs to note that $\rr^d$ itself corresponds to a space of probability distributions; as parameters of the natural exponential family of probability distributions with base measure $\mu_0$. For example, $\theta \mapsto N(\theta, I)$ embeds $\rr^d$ as the mean of a normal distribution. This corresponds to Example \ref{exmp:gaussian}. Hence, if we consider the convex set of all probability distributions that are absolutely continuous with respect to $N(0, I)$, we have an embedding of $\rr^d$ into that set. What we show is that there exists a concave function on this convex set of probability measures whose superdifferentials (taken in the usual way) are also probability measures that are members of the same exponential family. The parameters of these superdifferentials recover the optimal coupling for problem \eqref{eq:OTstatement}. 

This idea is closely related to our work \cite{PW16} where we describe this result only for Example \ref{exmp:simplex}. It leads to a new information geometry based on the concept of $L$-divergence that extends the classical information geometry of Kullback-Leibler divergence (that corresponds to Example \ref{exmp:gaussian}). See \cite{A16} for an introduction to information geometry. I expect a similar consequence for this general construction.

\section{Preliminaries}

\subsection{Natural exponential families}

We now define the exponential family of models with base measure $\mu_0$. See \cite{A16} for more details and historical references. Let $M_1:=M_1\left(\rr^d\right)$ be the space of Borel probability measures on $\rr^d$. 

\begin{defn}\label{defn:exp}
For every $\theta\in \rr^d$, define $\mu_\theta$ by the exponential change of measure 
\[
\frac{d\mu_\theta}{d\mu_0}(x)= \exp\left( \iprod{\theta, x} - \Lambda_0(\theta) \right).
\]
Then $\mu_\theta \in M_1\left(  \rr^d \right)$ and the collection $\left\{ \mu_\theta,\; \theta \in \rr^d \right\}$ is called the natural exponential family of models with base measure $\mu_0$. 
\end{defn}

For example, when the base measure is the multidimensional standard normal, for any $\theta$, the probability measure $\mu_\theta$ is the multidimensional Gaussian law with mean $\theta$ and identity covariance. Now, consider the exponential family generated by $\mu_0$ in Example \ref{exmp:simplex}. For any $\theta\in \rr^d$, notice that $\mu_\theta$ is still supported on the discrete set $\left\{  e_i,\; i=0,1,\ldots, d \right\}$. The mass it puts on $e_i$ is proportional to $\exp\left( \iprod{\theta, e_i}  \right)$. Thus, 
\[
\mu_\theta\left(  e_i \right)= \frac{e^{\theta_i}}{1+ \sum_{i=1}^d e^{\theta_i}}, \; \text{for $i=1,2,\ldots, d$, and} \quad  \mu_\theta\left(  e_0 \right)= \frac{1}{1+ \sum_{i=1}^d e^{\theta_i}}.   
\]
As $\theta$ varies in $\rr^d$ the natural exponential family corresponds to the open unit simplex of dimension $d$ (i.e., with $(d+1)$ coordinates). 

\begin{lemma}\label{lem:compu1} For any $\theta, \gamma \in \rr^d$, we have $\log\mu_{\theta}\left( e^{\iprod{\gamma, x}}  \right)=\Lambda_0(\gamma+ \theta) - \Lambda_0(\theta)$. 
\end{lemma} 

\begin{proof}
By a change of measure, $\log\mu_{\theta}\left( e^{\iprod{\gamma, x}}  \right)=\log \mu_0\left( e^{\iprod{\theta+\gamma,x}}  \right)- \Lambda_0(\theta)$. 
\end{proof}

\subsection{Topological preliminaries} To perform convex analysis on the infinite dimensional space of probability measures we require a proper locally convex topological vector space (LCTVS). Let $M\left( \rr^d\right)$ be the space of all finite signed measures on $\rr^d$ which have finite exponential moments of all orders. For $\nu \in M\left( \rr^d  \right)$, we will follow the usual notation of denoting by $\nu^+, \nu^-$, and $\abs{\nu}$, the positive part, the negative part, and the absolute value (or, variation) of the signed measure $\nu$. For $i \in [d]:=\{1,2,\ldots,d\}$ and $j\in \NN \cup \{0\}$, define the set of functions from $M_0\left( \rr^d \right)$:
\[
p^{(i)}_j(\nu)= \abs{\nu}\left(  e^{j \abs{x_i}} \right), \quad \nu \in \VS.
\]
Consider the vector space $\VS$ of elements $\nu$ in $M\left(  \rr^d \right)$ such that $p^{(i)}_j(\nu) < \infty$ for every $i \in [d]$ and $j \in \NN$. We now define a locally convex topology on $\VS$.  

\begin{lemma}
The family $\left( p^{(i)}_j,\; i\in [d], j\ge 0  \right)$ is a collection of seminorms that is total. Hence, it induces a locally convex Hausdorff metrizable topology on $\VS$. In this topology, we have $\lim_{n \rightarrow \infty} \nu_n = \nu$ if and only if 
\eq\label{eq:cvgexp}
\lim_{n \rightarrow \infty} \abs{\nu_n}\left(  e^{\iprod{\theta, x}}   \right) = \abs{\nu}\left(  e^{\iprod{\theta, x}}   \right), \quad \forall \; \theta \in \rr^d. 
\en 
Hence, when each $\nu_n$ is a probability measure, convergence in this topology is equivalent to the convergence of all exponential moments. 
\end{lemma}

\begin{proof} It is obvious that every $p_j^{(i)}$ is a seminorm. In fact, it is a norm, and thus, the family is total. Therefore, there is a corresponding locally convex Hausdorff topology  which is the smallest topology under which each $p_j^{(i)}$ is continuous. The topology is metrizable since the family of seminorms is countable.  

We now show \eqref{eq:cvgexp}. To see the \textit{only if} part, consider some $\theta\in \rr^d$ and let $j:= \lceil \norm{\theta}_1 \rceil$. Then 
\[
0\le e^{\iprod{\theta, x}} \le e^{j \max_i \abs{x_i} } \le \sum_{i=1}^d e^{j \abs{x}_i}.
\]
Thus the function $\nu \mapsto \abs{\nu}\left( e^{\iprod{\theta, x}}  \right)$ is continuous in the locally convex topology constructed above. which gives us the \textit{only if} part. 

For the \textit{if} part, fix $i \in [d]$ and $j \ge 0$, and define $\theta^+=j e_i$ and $\theta^-= -j e_{i}$. Now, 
\[
e^{j \abs{x}_i}  \le e^{\iprod{\theta^+, x}} + e^{\iprod{\theta^-, x}},
\]
whereby convergence of the right side gives convergence of the left. 
\end{proof}

We shall call this topology $\topo$. Consider the dual space $\VS^*$ of all linear $\topo$ continuous functions on $\VS$. Equip $\VS^*$ with the weak* topology $\topo^*$. This gives us a pair of locally convex topological vector spaces $\left(\VS, \VS^*\right)$ such that if $\nu \in \VS$ and $h \in \VS^*$ then the bilinear function $\iprod{h,\nu}:=\nu(h)$ is $\topo$ continuous on $\VS$ and $\topo^*$ continuous on $\VS^*$. 

It follows from \eqref{eq:cvgexp} that the function $x\mapsto e^{\iprod{\theta, x}}$, for any $\theta\in \rr^d$, can be thought of as an element in $\VS^*$. We will use this identification without further remark.

\section{Exponentially concave functions and $\Lambda_0$-concave functions} 

Let $\Omega$ be a convex subset of $M_1$. We generalize the definition of exponentially concave functions from \cite{P16, PW16} where the reader can find more references and applications to various other fields.

\begin{defn}
A function $\varphi:\Omega \rightarrow \rr \cup \{ -\infty\}$ will be called exponentially concave if $\exp\left( \varphi \right)$ is a nonnegative concave function on the convex set $\Omega$. In particular, $\varphi$ is itself concave. 
\end{defn}

Consider the base measure $\mu_0$ and let $\Omega_0$ denote the convex set of probability measures on $\rr^d$ that are absolutely continuous with respect to $\mu_0$. Let $\varphi$ denote an exponentially concave function on $\Omega_0$. We will throughout assume that $\varphi$ is proper and is continuous in the interior of its effective domain, which is non-empty. By \cite[Proposition 5.2]{ET76}, at any $\mu$ is the interior of the effective domain, the set of superdifferentials $\partial \varphi(\mu)$ is non-empty. Consider one such point $\mu \in \Omega_0$ and let $\mu^*\in \VS^*$ be an element in the superdifferential $\partial \varphi(\mu)$. 

Note that, since $\Omega_0$ is a set of probability measures, the set of superdifferentials is closed under addition by a constant. Thus if we replace $\mu^*$ by $\widetilde{\mu^*}:= \mu^* +1 - \iprod{\mu^*, \mu}$, then the latter is an element in $\partial \varphi(\mu)$ and satisfies $\iprod{\widetilde{\mu^*}, \mu}=1$. To keep our notations simple, we will assume that $\iprod{\mu^*, \mu}=1$. We will throughout use this normalization. 

Our next lemma shows that supergradients of exponential concave functions can be expressed in terms of a probability measure.  This generalizes \cite[Proposition 5]{PW14} where it is shown for the unit simplex (Example \ref{exmp:simplex}).

\begin{lemma}\label{lem:supdiff}
There exists an element $\pi=\pi(\mu) \in \Omega_0$ that induces $\mu^*$ in the following way:
\eq\label{eq:subgradprob}
\iprod{\mu^*, \nu}= \pi\left( \frac{d\nu}{d\mu}  \right), \quad \text{for all $\nu \ll \mu$}. 
\en
\end{lemma}

We will denote $\mu^*$ by $d\pi/d\mu$ in view of the above lemma. Consistent with the terminology developed in \cite{PW14, PW16} we will call the map $\pi:\Omega_0 \rightarrow \Omega_0$ to be a portfolio map generated by $\varphi$.

\begin{proof}[Proof of Lemma \ref{lem:supdiff}] Let $\Phi=\exp\left( \varphi \right)$. Then $\Phi$ is a nonnegative concave function on $\Omega_0$ and $\Phi(\mu)\mu^* \in \partial \Phi(\mu)$. Consider $\nu \in \Omega_0$. Since $\Phi \ge 0$ and concave, for $0< t < 1$, we get 
\eq\label{eq:nonneg1}
\begin{split}
-\Phi(\mu) &\le \Phi\left( \mu + t (\nu- \mu)   \right) - \Phi(\mu)\le \iprod{\Phi(\mu) \mu^*, t (\nu-\mu)}= t \Phi(\mu) \iprod{\mu^*, \nu - \mu}.
\end{split}
\en
Since $\Phi(\mu) > 0$, we divide both sides by it and take $t \uparrow 1$ to get 
\eq\label{eq:nonneg}
\iprod{\mu^*, \nu} \ge \iprod{\mu^*, \mu} - 1=0, 
\en
due to our chosen normalization. 

Consider the space of continuous functions supported on $B_K:=\{x:\; \norm{x} \le K \}$. Let $g$ be one such function such that $g \ge 0$ and $\mu(g)=1$. Then define $\mu_g \in \Omega_0$ by the change of measure $d\mu_g/d\mu= g$. Consider the map $g \mapsto \Gamma(g):=\iprod{\mu^*, \mu_g}$. It can be extended to all continuous $g$ supported on $B_K$ by using linearity. Since all our functions are supported on $B_K$, uniform convergence of functions imply convergence in the $\topo$ topology for the corresponding sequence of probability measures. Thus $\Gamma$ is a continuous map in the uniform topology that takes nonnegative functions to nonnegative values by \eqref{eq:nonneg}. By the Riesz representation theorem, there exists a nonnegative measure $\pi_K$ such that $\Gamma(g)=\pi_K(g)$. Consistency over $K$ gives us a $\pi$ whose restriction to $B_K$ is $\pi_K$. That $\pi$ is a probability is due to our normalization. This shows \eqref{eq:subgradprob} for $\nu=\mu_g$ for all bounded continuous $g$. 
Generalization to non-continuous $g$ with bounded support follows by uniform continuous approximation. 
If $g$ has unbounded support, consider a collection of increasing Borel maps $g_i,\; i\in \NN$, such that $\lim_{i\rightarrow\infty} g_i=g$. By monotone convergence theorem, $\lim_{i \rightarrow \infty} \mu_{g_i}= \mu_g$ in the $\topo$ topology.  
This proves \eqref{eq:subgradprob} for all $\nu \ll \mu$. 
\end{proof}

We now consider a special class of exponentially concave functions. Let $h:\rr^d \rightarrow [-\infty, \infty)$ be a measurable function. Define $\varphi$ on $\Omega_0$ by 
\eq\label{eq:expcvexp}
\varphi(\mu)=\inf_{\theta \in \rr^d}\left[ \log  \mu\left(  e^{\iprod{-\theta, x}} \right) - h(\theta)  \right], \qquad \mu \in \Omega_0. 
\en
Clearly, $\varphi$ is exponentially concave. Our next result connects exponentially concave functions with $\Lambda_0$ concave functions. Recall the natural exponential family with base measure $\mu_0$. Notice that the natural exponential family with base measure $\mu_0$ is a subset of $\Omega_0$. 

\comment{
Consider $\pi(\cdot)$ from Lemma \ref{lem:supdiff}. Define $\pi_\theta:= \pi_{\mu_\theta}$ and recall that $\pi_\theta(x)$ is the mean under $\pi_\theta$.  
}

\begin{thm}\label{thm:expcnvtocost} Consider $\varphi$ from \eqref{eq:expcvexp}. Define the function $\psi: \rr^d \rightarrow \rr$ by
\eq\label{eq:l0cv}
\psi(\alpha)= \varphi(\mu_\alpha) + \Lambda_0(\alpha), \quad \alpha \in \rr^d.
\en
Then $\psi(\cdot)$ is $\Lambda_0$ concave on $\rr^d$. Conversely, suppose that a $\Lambda_0$ concave function $\psi$ is given. Let $\psi^0$ denote its $\Lambda_0$ conjugate. Define an exponentially concave function on $\Omega_0$ by 
\eq\label{eq:defineexpcnv}
\varphi(\mu):= \inf_{\theta \in \rr^d}\left[  \log \mu\left( e^{-\iprod{\theta, x}}\right) - \psi^0(\theta)  \right], \qquad \mu \in \Omega_0. 
\en
Then, for all $\alpha\in \rr^d$, we have 
\eq\label{lem:phiatalpha} 
\varphi(\mu_\alpha)=\psi(\alpha) - \Lambda_0(\alpha).
\en
\end{thm}

\comment{
Consider the function $\omap\left(\theta\right) := \theta - \left(\nabla \Lambda_0\right)^{-1}\left( \pi_\theta\left(x\right)\right)$. Then the map $\theta \mapsto \omap(\theta)$ is a $\Lambda_{0}$ cyclically monotone map on $\rr^d$.}

\begin{proof} The first part follows by direct evaluation. For any $\alpha \in \rr^d$, we have 
\eq\label{eq:valueatalpha}
\begin{split}
\varphi(\mu_\alpha)&:= \inf_{\theta \in \rr^d} \left[  \log \mu_\alpha\left( e^{-\iprod{\theta,x}}\right) - h(\theta) \right]\\
&=\inf_{\theta \in \rr^d} \left[  \Lambda_0(\alpha - \theta) - h(\theta) \right] - \Lambda_0(\alpha). 
\end{split}
\en
Thus $\psi(\alpha)=\varphi(\mu_\alpha) + \Lambda_0(\alpha)$ satisfies the representation in Definition \ref{defn:c-concave}. 
The converse is a consequence of \eqref{eq:valueatalpha} replacing $h$ by $\psi^0$.
\end{proof}

We now show that the $\Lambda_0$ gradient of $\psi$ is related to the portfolio map of $\varphi$. Let $(V, V^*)$ be any pair of LCTVS and its topological dual. For each $\theta$ in some collection $\Theta$, let $f_\theta: V \rightarrow \rbar$ be a proper concave function. Consider the proper concave function $F: V \rightarrow \rbar$ given by $F(v):= \inf_{\theta \in \Theta} f_\theta(v)$. Let $\partial f(u)$ refer to the set of superdifferentials of a concave function $f$ at a point $u$. 

\begin{lemma}\label{lem:diffident}
For $v_0 \in \dom(F)$, $\partial f_\theta(v_0)\subseteq \partial F(v_0)$, for any $\theta$ such that $F(v_0)= f_\theta(v_0)$.   
\end{lemma}

\begin{proof} The proof is trivial. Pick any $v_0^*\in \partial f_\theta(v_0)$ and any other $u \in V$. Then, by the superdifferentiability of $f_\theta$ we get 
\[
F(u) \le f_\theta(u) \le f_{\theta}(v_0) + \iprod{v_0^*, u-v_0} = F(v_0) + \iprod{v_0^*,u-v_0}.
\]
Varying $u\in V$ proves the claim. 
\end{proof}

\begin{thm}\label{thm:cakederiv}
Consider the exponentially concave function \eqref{eq:defineexpcnv}. For $\alpha \in \rr^d$, let $\theta$ be an element in $\rr^d$ such that $\left( \alpha, \theta  \right) \in \partial^{\Lambda_0}\psi$. Then, there exists an element $\mu_\alpha^* \in \VS^*$ such that $\mu^*_\alpha$ is a supergradient of $\varphi$ at $\mu_\alpha$ and for any $\nu \in \VS$, we have $\iprod{\mu_\alpha^*, \nu}=\nu\left( h_\alpha(x) \right)$, where 
\[
h_\alpha(x):=\exp\left(  -\iprod{\theta, x} - \Lambda_0(\alpha- \theta) + \Lambda_0(\alpha)    \right). 
\]
In particular, the probability measure $\pi_\alpha:=\pi\left( \mu_\alpha\right)$ in Lemma \ref{lem:supdiff} is given by $\mu_{\alpha- \theta}$. 
\end{thm}

\begin{proof} We apply Lemma \ref{lem:diffident} to the choice: $f_\theta(\mu)=  \log \mu\left( e^{-\iprod{\theta, x}}\right) - \psi^0(\theta)$, $\theta \in \rr^d$. By \eqref{lem:phiatalpha} and the assumption that $(\alpha, \theta)\in \partial^{\Lambda_0}\psi$, we have
\[
\begin{split}
\varphi\left( \mu_\alpha  \right)&= \psi(\alpha) - \Lambda_0(\alpha)= \Lambda_0(\alpha-\theta) - \Lambda_0(\alpha) - \psi^0(\theta)\\
&= \log\mu_\alpha\left( e^{-\iprod{\theta, x}} \right) - \psi^0(\theta)=f_\theta\left( \mu_\alpha\right). 
\end{split}
\]
The first equality in the second line in the above display is due to Lemma \ref{lem:compu1}. From \eqref{eq:defineexpcnv} we get $\varphi(\mu)=\inf_{\theta}f_\theta(\mu)$. However from the above display, we also get $F(\mu_\alpha)=f_\theta(\mu_\alpha)$ for each $(\alpha, \theta)\in \partial^{\Lambda_0}\psi$. 

By Lemma \ref{lem:diffident}, it is sufficient to show the existence of the superdifferential with the claimed properties for $f_\theta$. In fact, we find the G\^{a}teaux derivative of $f_\theta$ at $\mu_\alpha$. To do this, fix $\nu \in \Omega_0$ and consider the family of probability measures $\mu(t)=(1-t)\mu_\alpha + t \nu$, for $0\le t \le 1$. Clearly $\mu(0)=\mu_\alpha$. We are interested in the limit:
\[
f_\theta'\left( \mu_\alpha  \right):=\lim_{t\rightarrow 0+} \frac{f_\theta(\mu(t)) - f_\theta(\mu_\alpha)}{t}=\frac{d}{dt}\Big\vert_{t=0+} \log \left[ \mu(t)\left(  \exp\left( -\iprod{\theta, x} \right)\right)\right].
\]
The last expression, via differentiation within the expectation, gives us
\[
\frac{\nu\left( \exp\left( -\iprod{\theta, x}  \right)  \right)}{\mu_\alpha\left( \exp\left( -\iprod{\theta, x} \right) \right)}-1 = \nu\left( h_\alpha(x) \right) -1.
\]
That proves that $f_\theta$ is G\^{a}teaux differentiable at $\mu_\alpha$ with the above derivative. 

Notice that $\mu_\alpha\left( h_\alpha(x)  \right)= 1$ by Lemma \ref{lem:compu1}. Hence $\nu\left( h_\alpha(x) \right) -1= (\nu- \mu_\alpha)\left( h_\alpha(x)  \right)$. Since $t\mapsto f_\theta(\mu(t))$ is concave, we also get the following expression of a supergradient:
\[
f_\theta(\nu) \le f_\theta(\mu_\alpha) + \left(\nu - \mu_\alpha\right)\left( h_\alpha(x)  \right).
\]
By Lemma  \ref{lem:diffident} this proves the claim regarding the existence of the supergradient of $\varphi$ at $\mu_\alpha$. The formula for $\pi_\alpha$ follows immediately from the Radon-Nikodym derivative 
\[
\frac{d\pi_\alpha}{d\mu_\alpha}(x)= h_\alpha(x),
\]
and an application of Lemma \label{lem:supdiff}.
\end{proof}

Thus the parameter of the portfolio map at $\mu_\alpha$ is $\alpha-\theta$, where $(\alpha, \theta)\in \partial^{\Lambda_0} \varphi$. Hence, as mentioned in the Introduction, if the support of the optimal transport includes the point $(\alpha, \theta)$, we see it embedded as the pair of elements $\left( \mu_\alpha, \mu_{\alpha-\theta}  \right)$ in the subdifferential of a true concave function on the space of probability measures. 

\section*{Acknowledgement} I am indebted to Prof. B. V. Rao for teaching me probability. He remains one of the finest teachers I have encountered in my life and interactions with him are some of my fondest memories from ISI. I wish him all the best on the occasion of his $70$th birthday. I also thank an anonymous referee for useful comments.

\bibliographystyle{alpha}

\bibliography{infogeo}

\end{document}